\documentclass{amsart}

\usepackage{amssymb}
\usepackage[stretch=10,shrink=10]{microtype}
\usepackage[showframe=false, margin = 1.5 in]{geometry}
\usepackage{mathtools}
\usepackage[shortlabels]{enumitem}
\usepackage{theoremref}
\usepackage{tikz}
\usepackage{subfigure, float}
\usepackage{pgfplots}
\usepackage{comment}
\usepackage{mathrsfs}
\usepackage{stmaryrd}

\renewcommand{\emptyset}{\varnothing}

\newcommand{\E}{\mathbf{E}}

\renewcommand{\vec}{\mathbf }

\newcommand{\f}{\frac}
\newcommand{\ind}[1]{\mathbf{1}{\{ #1 \}}}

\makeatletter
\newcommand*\proc{{\mathpalette\bigcdot@{.7}}}
\newcommand*\bigcdot@[2]{\mathbin{\vcenter{\hbox{\scalebox{#2}{$\m@th#1\bullet$}}}}}
\makeatother

\newtheorem{thm}{Theorem}
\newtheorem{lemma}[thm]{Lemma}
\newtheorem{prop}[thm]{Proposition}

\newtheorem{conjecture}{Conjecture}

\theoremstyle{remark}
\newtheorem{remark}[thm]{Remark}
\theoremstyle{definition}

\newcommand{\meet}{\leftrightarrow}

\DeclareMathOperator{\BA}{BA}

\newcommand{\X}{\vec X}

\title{The upper threshold in ballistic annihilation}

\author{Debbie Burdinski}
	\address{Debbie Burdinski}
\email{\texttt{debbie.burdinski@duke.edu}}	
	
\author{Shrey Gupta}
\address{Shrey Gupta}
	\email{\texttt{shrey.gupta@duke.edu}}

\author{Matthew Junge}
	\address{Matthew Junge}
	\email{\texttt{jungem@math.duke.edu}}

\begin{document}
\maketitle

\begin{abstract}
Three-speed ballistic annihilation starts with infinitely many particles on the real line. Each is independently assigned either speed-$0$ with probability $p$, or speed-$\pm 1$ symmetrically with the remaining probability. All particles simultaneously begin moving at their assigned speeds and mutually annihilate upon colliding. Physicists conjecture when $p \leq p_c = 1/4$ all particles are eventually annihilated. Dygert et.\ al.\ prove $p_c \leq .3313$, while Sidoravicius and Tournier describe an approach to prove $p_c \leq .3281$. For the variant in which particles start at the integers, we improve the bound to $.2870$. A renewal property lets us equate survival of a particle to the survival of a Galton-Watson process whose offspring distribution a computer can rigorously approximate. This approach may help answer the nearly thirty-year old conjecture that $p_c >0$. 

\end{abstract}

\section{Introduction}

Two decades ago, physicists devoted considerable attention to a simple but difficult to analyze process called \emph{ballistic annihilation} \cite{ b5,  b3, b1, b4,b11, b8, b13, b14}. Particles are placed on the real line according to a unit intensity Poisson point process and each is independently assigned a speed according to a probability measure $\nu$. After this assignment, the model is deterministic; particles move at their speed and mutually annihilate upon colliding. The canonical example is with speeds from $\{-1,0,1\}$ sampled according to the symmetric measure
\begin{align}
\nu =  \f{1-p}2 \delta_{-1} + p \delta_0 + \f{1-p}2 \delta_1.\label{eq:nu}
\end{align}
%The big question for ballistic annihilation with measure $\nu$, that will be contextualized momentarily, is the longtime density of particles.

Physicists refer to this as an $A + A \to 0$ process. Ballistic annihilation was introduced in \cite{b5} with just two speeds. The goal was to study interactions of ideal gas particles. A few years later it was discussed in the context of arbitrary continuous $\nu$ \cite{b2}. The authors' motivation was to understand ``intriguing features" in the decay kinetics of \emph{irreversible aggregation}, $A_i + A_j \to A_{i+j}$, which has been used to model coalescence of fluid vortices 
\cite{fluid} and planet formation by accretion \cite{ planet}. The authors give heuristics for the decay rate of particles, and conjecture it responds continuously to perturbations in the speed measure. 

The followup work \cite{b8} predicts more interesting behavior in ballistic annihilation with discrete speeds. It is thought that the process undergoes an abrupt phase transition. Consider ballistic annihilation with the measure from \eqref{eq:nu}.  We will call speed-$0$ particles \emph{inert} and speed-$\pm1$ particles  \emph{active}.
Let $\theta_t(p)$ be the probability an inert particle survives up to time $t$, and $\theta(p) = \theta_\infty(p)$ be the probability it is never annihilated. 
Krapivsky et.\ al.\ infer in \cite{b8} that there is a critical value $p_c$ such that $\theta(p)=0$ for $p\leq p_c$, and above $p_c$  it holds that  $\theta(p) >0$. They conjecture $p_c=1/4$ and make precise 
predictions for the behavior of $\theta_t$:

\begin{align}
\theta_t(p) \sim \begin{cases} C_pt^{-1}, & p < p_c \\ Ct^{-2/3}, & p = p_c \\ 2 - p^{-1/2}, & p > p_c \end{cases}, \qquad \text{ with } \qquad  C_p = \frac{ 2p}{ (1- 4p) \pi} \text{ and } C = \frac{2^{2/3}}{ 4\Gamma(2/3)^2 }. \nonumber %\label{eq:theta}
\end{align}
\quad \\
%For $p<p^*$ they conjecture the $p_t \approx t^{-1}$ and for $p >p^*$ we have $p_t \approx C_p>0$, and so a given 0 survives forever with positive probability 
% At criticality, they conjecture 0's do not survive, and decay like $p_t \approx t^{-2/3}$. 
% Here are simulations  just above, and just below criticality:
%\begin{figure}
%	\includegraphics[width = .9 \textwidth]{ba25}
%%	\includegraphics[width = .7 \textwidth] {ba24.png}
%%	\includegraphics[width = .7 \textwidth] {ba26.png}
%	\caption{Ballistic annihilation with  $p=1/4$.
%	% (middle) $p=.24$, and (bottom) $p=.26$. We see far fewer surviving 0-speed particles for $p \leq p_c$. 
%	Figure from \cite{arrows}.} \label{fig:ba2426}
%\end{figure}

%Figure \ref{fig:ba2426} illustrates what the process looks like at this conjectured critical value.
  A simple heuristic is given in \cite{b4} for why $p_c=1/4$. Active particles move towards one another at relative speed 2, while the gap between a moving and inert particle is covered at relative speed 1. Thus, collisions between active particles ought to occur twice as often as those between inert and active particles. 
If we look at all of the collisions in a large interval, then each is one of three possibilities:
\begin{align}(0,-1), (1,0), \text{ and } (1,-1). \label{eq:heu} 
\end{align}
Doubling the $(1,-1)$-collisions to account for the prediction that these occur twice as often, we expect on average that for every eight particles removed, two are inert particles. So, when $p=1/4$ the collision types balance. 

Symmetry ensures that active particles are annihilated almost surely (see \cite[Proposition 16]{bullets}.) Note that \cite{b4} also provides a description of the decay density of $\pm1$-speed particles in the different regimes of $p$. For $p<1/4$ the right tail is predicted to have exponent $-1/2$. At criticality it is claimed to be $-1$, and for $p>p_c$ they infer that the survival time decays at an exponential rate.
Droz et.\ al.\ in \cite{b4} provide a nearly complete proof of these conjectures (including those for $\theta_t(p)$ predicted by \cite{b8}). However, some steps are not  rigorous and the argument gives little intuition. These formulas come from a complicated differential equation involving the distance between neighbor particles at time $t$. Krapivsky et.\ al.\ point out in \cite{b8} that there is still need for methods ``that would provide better intuitive insights into the intriguing qualitative features of ballistic annihilation." For example, there is no probabilistic proof that $p_c \neq 0$, let alone is equal to $1/4$. 
%Showing $p_c$ is positive is a notorious question (asked in \cite{bullets,arrows, b8})

As for upper bounds, Dygert et.\ al.\ prove  that $p_c \leq .3313$ \cite{bullets}. 
%Note that the proof in \cite{bullets} is for unit spacings, but the authors remark that the proof also works with exponentially distributed gaps. 
Additionally, \cite{arrows} considers ballistic annihilation. Sidoravicius and Tournier prove that $p_c \leq 1/3$, and outline an approach to prove $p_c \leq 0.3280$. Their proof, like the one contained in this work, is recursive in nature. The configuration of particle speeds is revealed in blocks, and the number of surviving particles can be estimated by the subconfiguration in these blocks. The main difference is that we focus on the number of surviving inert particles at a special renewal time (\thref{prop:eta}). This yields a necessary and sufficient condition for survival in terms of a Galton-Watson process (\thref{prop:iff}). Moreover, because we consider unit-spacings between particles, it is more computationally tractable to estimate the offspring distribution of the Galton-Watson process (\thref{thm:pc}). 
%Laurenzo Taggi also claims to have an unpublished argument that improves to $p_c \leq .31$. 
We are not trying to start a hunt for bounds closer and closer to the conjectured value. Rather, by improving the known bound, we reveal a new perspective into what promotes survival of speed-0 particles. Ideally this may help prove the central question in ballistic annihilation, that, for small $p$, inert particles do not survive.

%\subsection{Notation and theorem}
In this article, we consider ballistic annihilation with an inert particle at the origin, and particles placed on $\mathbb Z^+$ with i.i.d.\ speeds sampled according to $\nu$ from \eqref{eq:nu}.  To distinguish our case from ballistic annihilation with exponentially distributed spacings, we define $\psi(p)$ as the probability the speed-0 particle at the origin is never annihilated, and
$$p'_c = \inf\{ p \colon \psi(p) > 0\}.$$
Understanding survival of inert particles appears to be equally interesting and challenging whether the spacings are  deterministic, or exponential(1) distributed. This is supported by recent findings \cite{bullets2} for the closely related bullet process with finitely many particles and a non-atomic  probability measure on speeds. They find that the law for the number of surviving particles is independent of the initial spacings. 
%Another recent article considers a self-sorting 
Our main result is an improved bound for $p_c'$.

\begin{thm} \thlabel{thm:pc}
$p_c' \leq 0.2870$.
\end{thm}

For ballistic annihilation with unit spacings, triple collisions may slightly change the critical threshold. Let $q= (1-p)/2$, so that the probability three consecutive particles triple collide is  $pq^2$ (i.e.\ a $(1,0,-1)$ configuration). This is $\approx .07$ when $p = 1/4$. With exponential spacings, the process loses one inert and one active particle whenever this configuration  occurs. However, with the triple collision, an extra active particle is destroyed. The heuristic at \eqref{eq:heu} predicts that, after one time unit, the density of $0$-particles is
		$$w = p-2pq+pq^2.$$
The $2pq$ term accounts for the configurations $(1,0)$ and $(0,-1)$. We add $pq^2$ to prevent double counting triple collisions. Similarly, the density of active particles is 
		$$z = 2 (q -q^2 -p q).$$ 
The heuristic at \eqref{eq:heu} suggest we solve $ w = \f 1 4 z$, which yields $p \approx .2450.$ Accordingly, we conjecture that $p_c' < .2450 < p_c.$ We do not think that $p_c'$ is much smaller than $.2450$.  The next spacing at which a triple collision can occur is $(1,\; \cdot \;,\; \cdot\;, 0 , \;\cdot\;, \;\cdot\;, -1)$.  Such configurations have probability described by a degree-7 polynomial in $p$ and $q$, and so the contribution will be very small.

%\begin{figure} 
%	\includegraphics[width = 15 cm]{pc} 
%	\caption{The $y$-axis is the ending proportion of surviving inert particles and the $x$-axis is $p$. The red line is for $10^5$ particles averaged over $1000$ simulations, and the blue for $10^6$ averaged over $100$ simulations. Simulations were provided by Laurent Tournier.} \label{fig:pc'}
%\end{figure}
%

  %the critical threshold is smaller in ballistic annihilation with unit spacings. \HOX{Can we make a good prediction for how much smaller? Perhaps by estimating the proportion of triple collisions that should occur... -MJ}
%Borrowing from the heuristic at \eqref{eq:heu} Thus, we leave it as a conjecture that $p_c = p_c'$

We prove this by relating survival in ballistic annihilation to survival of a Galton-Watson process. The idea is that there is a random index $\eta$ for which (i) when restricted to just particles in $[0,\eta]$ only inert particles survive, and (ii) the speeds of particles beyond $\eta$ are independent. If $Z$  is the number of surviving inert particles in the process restricted to $[0,\eta]$, then (ii) guarantees that each of these particles will independently spawn $Z$-distributed more surviving inert particles. Because each new generation is determined independently, we obtain a Galton-Watson process that counts surviving particles. 

The construction of $\eta$ is described in the proof of \thref{prop:eta}, and the equivalence of $\psi(p) > 0$ to survival of the induced Galton-Watson process is contained in \thref{prop:iff}. This characterization
\begin{align}\psi(p) >0 \iff \E Z >1 \label{eq:iff}
\end{align}
is novel and  could be useful for showing a non-survival regime for inert particles. Note that we suppress the $p$-dependence in the expectation and write $\E Z$ in place of $\E_p Z$. 
The same reasoning we use to obtain \eqref{eq:iff} can be applied to the usual ballistic annihilation with exponential spacings. Thus, a similar equivalence holds for $\theta(p)$. In \thref{remark}, we briefly sketch how to adapt \thref{prop:eta} to this setting.

\begin{comment}
Combining this with the conjectured phase transitions for $\theta_t(p)$ suggests a dramatic discontinuity for $\E Z$.

\begin{conjecture}  $\E Z= \begin{cases} \infty, & p > p_c' \\ 1,& p =p_c' \end{cases},$ and $\E Z < 1$ when $p < p_c'$.
	%when $p > p_c'$ and $\E Z \leq 1$ for $p \leq p_c'$. Moreover, \end{conjecture}
\end{conjecture}
 
%\noindent The same abrupt phase transition occurs for the parking process on Poisson(1)-Galton-Watson trees \cite[Theorem 1.2]{tree}. This is a related annihilating system with active particles (cars) and inert particles (parking spots). We do not see any obvious connection that would explain why the expectation has a similar abrupt discontinuity.
It may be surprising how challenging it is to calculate that $\E Z>1$ when we conjecture $\E Z = \infty$. However, it seems a significant contribution to $\E Z$ comes from a $1$-particle surviving. Although it is conjectured that the survival time of a $1$-particle has exponential tail, the expected survival time is also conjecture to diverge as $p \to p_c$. As our approach requires that we enumerate all possible configurations, this quickly becomes intractable for $p$ near $p_c'$.
\end{comment}

Once we have the equivalence at \eqref{eq:iff}, we turn our attention to lower bounding $\E Z$. There are certain random times before $\eta$ at which we can stochastically lower bound $Z$. We then use a computer to estimate these probabilities. This is completely rigorous, but requires too many calculations to be done by hand.

%Our approach is only limited by computing power in the following sense.
%
%\begin{thm}\thlabel{thm:comp}
%For all $\epsilon>0$ there exists a finite calculation that proves $p_c > p^* -\epsilon.$ 
%\end{thm}

\subsection*{Notation}

%We will refer to particles with speed 0 as \emph{inert}. 
Depending on how specific we need to be, we will refer to particles with speeds $\pm1$ as either \emph{active} or as \emph{$\pm 1$-particles}. Because it will correspond to the original parent in a Galton-Watson process, we refer to the inert particle at the origin as the \emph{seed}. The randomness in ballistic annihilation is an initial vector $\vec X = (X_i)_{i=0}^\infty$ with $X_0 =0$ and the $X_i$ for $i \geq 1$ are i.i.d.\ with law $\nu$ from \eqref{eq:nu}. These are the particle speeds. 
%Let $\F$ be the sigma algebra generated by $\vec X$, and $\F_n$ the sigma algebra generated by $(X_i)_{i=0}^n$. 

 Let $a_i$ denote the particle initially at $i$.   We will let $a_i \meet a_j$ mean that particles at $i$ and $j$ mutually annihilate. With discrete speeds and spacings, it is possible that three particles collide simultaneously. Denote this with $a_i \meet a_j \meet a_k$. We emphasize that writing $a_i \meet a_j$ does not preclude a third particle also being annihilated. That is $\{a_i \meet a_j\} \cap \{a_i \meet a_j \meet a_k\} \neq \emptyset.$ 
We will sometimes use the more specific notation $a_i \mapsto a_j$ for when the active particle from $i$ annihilates the inert particle at $j$.

\section{An embedded Galton-Watson process}

To show that inert particles survive with positive probability, it suffices to establish that the seed is never annihilated with positive probability. This is because an inert particle, say $a_n$, has $n$ particles to its left, and thus has some positive probability of not being annihilated from that side. For example, all left particles are inert with probability $p^{n-1}$. Conditional on surviving from the left, the probability that $n$ is never annihilated by a particle from the right is the same as the probability the seed survives. This follows via the coupling that aligns the speeds to the right of the seed with those to the right of $a_n$ in two independent processes.

That said, in this section, we develop a framework that lets us relate survival of the seed to non-extinction of a Galton-Watson process that counts surviving inert particles. This is made explicit in \thref{prop:iff}. The construction hinges upon a renewal structure that has nice monotonicity properties. It rests upon the following two lemmas. 
\begin{lemma} \thlabel{lem:active}
Let $i<j$ and $X_j = -1$. The random variables $(X_{k})_{k >j}$ are $\nu$-distributed and independent of the event $\{a_i \meet a_j\}$.	
\end{lemma}

\begin{proof}

The particles to the right of $a_j$ cannot influence the event $\{a_i \meet a_j\}$, thus the speeds are independent of the event.
\end{proof}
If an active particle destroys an inert particle, then this induces a short range dependence. We know that all of the active particles that could reach the inert particle before its destroyer arrives must be annihilated. However, if we look sufficiently far away the particle speeds are once again independent. 
\begin{lemma} \thlabel{lem:inert}
Let $i <j$ and $X_j =0$. The random variables $(X_k)_{k>j +(j-i)}$ are independent of the event $\{a_i \mapsto a_j\}$.  	
\end{lemma}
\begin{proof}
It is elementary to work out that the particles beyond $j+(j-i)$ cannot reach $a_j$ before $a_i$ does, thus their speeds are independent of the event.
\end{proof}

\begin{remark}
The speed $X_{j+(j-i)}$ is independent of $\{a_i \mapsto a_j\}$ but possibly $a_{j + (j-i)}$ triple collides with $a_j$. 
\end{remark}

Conditioned on the event from \thref{lem:inert} we will refer to $[i,2j-i]$ as a \emph{window of dependence}. Given $\vec x = (x_0,x_1,\hdots, x_n)$, we define $\BA(\vec x)$ to be ballistic annihilation on $\mathbb R$ with particles at $0,1,\hdots,n$ where the particle at $i$ has speed $x_i$. 
Run $\BA(\vec x)$ until every collision that could occur has occurred (this takes at most $n$ time units). Let 
\begin{align}\xi_i(\vec x) = x_i\ind{a_i \text{ survives}} + 2\cdot \ind{a_i \text{ is annihilated}}\label{eq:xi}
	\end{align}
 and 
$\xi(\vec x) = ( \xi_0,\hdots, \xi_n)$
so that the $i$th entry is the speed of $a_i$ if it survives, or $2$ if $a_i$ is annihilated in $\BA(\vec x)$. 

We return to ballistic annihilation with the seed at the origin and a particle with a $\nu$-distributed speed at each nonnegative integer. The last piece of notation we need is that $\vec X[ i,j] = (X_i,\hdots, X_j)$ is the restriction to the coordinates between $i$ and $j$. We now explain the renewal structure. The following proposition asserts that there exists a random index $\eta$ such that only inert particles survive in $\BA(\X[0,\eta])$, and the particle speeds beyond $\eta$ are independent.

\begin{prop} \thlabel{prop:eta}
There exists a random variable $\eta\geq 1$ such that either 
\begin{enumerate}[label = (\roman*)]
	\item  $\eta = \infty$ and the seed survives, or 
	\item $\eta < \infty$ and $\xi(\vec X[0,\eta]) \subseteq \{0,2\}^\eta$ and $(X_i)_{i > \eta}$ are $\nu$-distributed and independent of $\eta$ and each other. Additionally, if $\eta >1$, then the seed survives in $\BA(\vec X[0,\eta])$. 	
\end{enumerate}
\end{prop}

\begin{proof}

We remark that the last condition of (ii) for $\eta>1$ is included to exclude the trivial renewal time at the index of the $-1$-particle that destroys the seed. So, unless $\eta=1$, the particle at $\eta$ will not destroy the seed. 

We define $\eta$ in terms of $X_1$ and a random variable $\eta_1$ for the distance we must look out for the process to be completely independent of how $a_1$ is annihilated when $X_1 =1$:
\begin{align}
\eta := \ind{X_1\neq 1} + \ind{X_1 = 1} \eta_1. \label{eq:eta}
\end{align}
Now we describe $\eta_1$. Consider ballistic annihilation with $X_0 =0$ and $X_1 =1$, and let $\gamma_1$ be the index of the particle $a_{1} \meet a_{\gamma_1}$. First off, if $\gamma_1 = \infty$ then set $\eta_1=\infty$ and the seed survives. Supposing $\gamma_1 < \infty$, if there is a triple collision, we take the larger of the two indices (so necessarily $X_{\gamma_1} = -1$ in this case). If $X_{\gamma_1} = -1$, then we set $\eta_1 = \gamma_1$. In this case, we have $\xi(\X[0,\eta_1]) = \{0\} \times \{2\}^{\eta_1}$, and the particle speeds beyond $\eta_1$ are independent because of the renewal in \thref{lem:active}.

It gets more complicated when $X_{\gamma_1} = 0$. \thref{lem:inert} with $i=1$ and $j= \gamma_1$ tells us that there is an $I_1:=[1,2 \gamma_1 - 1]$ window of dependence on this event. Let $\X_1 = \X[0,2 \gamma_1 -1]$ and $$\tau_1 = \max \{ i \in I_1 \colon \xi_i(\X_1) = 0\}$$ be the starting location of the rightmost surviving inert particle. Note that the seed must survive, so we know that $\tau_1$ is well defined and nonnegative. Also, let $$\kappa_1 = \min \{ \tau_1 \leq i \leq 2\gamma_1 - 1\colon \xi_i (\mathbf X_1) = 1\}$$ be the first surviving $1$-particle to the right of $\tau_1$. If there is no such particle, set $\kappa_1 = 0.$ We will consider this case separately in a moment.

If $\kappa_1 >0$, then we let $\gamma_2$ be such that $a_{\kappa_1} \meet a_{\gamma_2}$. In words, $\gamma_2$ is the index of the particle that destroys the active particle $a_{\kappa_1}$.  As before, if there is a triple collision, then we take the particle with larger index. If $X_{\gamma_2} = -1$, then the process renews and we set $\eta_1 = \gamma_2$. If not, then a new window of dependence is induced by the event $\{a_{\kappa_1} \mapsto a_{\gamma_2}\}$. We once again look at the furthest left surviving $1$-particle and either set $\eta_1$ equal to the index of the $-1$-particle that destroys it, or iterate another step if the particle hits an inert particle, causing a new window of dependence. This will either: halt eventually, giving some value of $\eta_1$; never halt, in which case $\eta_1 =\infty$; or give some $\kappa_i = 0$, which puts us in the case described now for $\kappa_1$. 

If $\kappa_1 = 0$, then there are no surviving active particles in $\BA(\X_1)$. However, it is possible that some $1$-particles that annihilate inert particles in $\BA(\X_1)$ will survive longer when  particles from beyond $\X_1$ are introduced. For example, the survival of a $1$-particle annihilated by an inert particle at $x$ is prolonged if a $-1$-particle started outside of $\X_1$ reaches $x$ first. To account for this, we look at the collection of $1$-particles from $\BA(\X_1)$ that are destroyed by inert particles
$$J = \{ j \colon X_j \in \X_1, X_j =1, a_j \mapsto a_{j'} \text{ with } X_{j'} = 0\}.$$
If $J$ is empty, then we set $\eta_1 = 2\gamma_1-1$. Otherwise, for each $j \in J$ we augment $\X_1$ to a larger interval $\X_{1,j}= \X[w_j,z_j]$ defined to be the  interval such that a $-1$-particle at $z_j$ could triple collide (if it had a clear path) with $a_j$ and $a_{j'}$. Set $z = \max_{j \in J} z_j$ to be the furthest right index that we must consider in order to know whether each collision $a_{j} \mapsto a_{j'}$ occurs in the $\BA(\X)$. 

If all of the collisions counted by $J$ occur, then we possibly gain some additional particles in $\BA(\X[0,z])$. However, since all of the collisions in $J$ still occurred we once again cannot have any surviving $-1$-particles in $\BA(\X[0,z])$. So, we repeat the procedure from some $\kappa_2$ equal to the leftmost surviving $1$-particle in $\BA(\X[0,z])$. This will either eventually terminate and we will obtain $\eta$, or never halt and set $\eta = \infty$ with the seed never being reached by a $-1$-particle. 

The last case to consider is that some collision counted by $J$ does not occur because a particle from $[2 \gamma_2 - 1,z]$ reaches an inert particle first. Letting $\kappa_2$ be the leftmost such $1$-particle, we then know that $a_{\kappa_2}$ will now survive longer when we introduce more particles to the right. We set $\gamma_2$ to be the index such that $a_{\kappa_2} \meet a_{\gamma_2}$. This reinitiates the procedure we have defined, which will either terminate at a renewal time $\eta$, or never terminate so that $\eta = \infty$ and the seed survives.
%Triple collisions create an issue here because an inert particle in $I_1$ that is destroyed in $\BA(\X_1)$ might be involved in a triple collision with a $-1$-particle outside of $I_1$. To account for this we look at all collisions of $1$- and inert particles in $\BA(\X_1)$ and set $I'_2= [w,z]$ to be the largest window of dependence induced by these. If $I'_2  \subseteq I_1$ then there will be no interference of the independent particles outside of $I_1$ with those inside $I_1$. In this case we set $\eta_1 = 2 \gamma_1 -2$. If $I'_2 \subsetneq I_1$ then we consider $\BA( \X[0,z])$. The same reasoning as before about windows of dependence ensures that we only obtain more surviving inert particles  in $\BA(\X[0,z])$ then in $\BA(\X_1)$. After running it, we will be in one of the three cases just discussed. So, either we will obtain some $\eta_1$, or the process continues indefinitely and $\eta_1 = \infty.$ 
%
\end{proof}

\begin{remark} \thlabel{remark}
A similar statement also holds when particles are placed according to a unit intensity Poisson process. \thref{lem:active} and \thref{lem:inert} still hold in this setting, and thus we can follow the same steps to obtain a renewal as in \thref{prop:eta}. 
\end{remark}

%\begin{proof}
% Let $\eta_0 =1$ and $\xi^{(0)} = \xi(\X[0,\eta_0])$. Set $\kappa_0 = \min \{i \colon \xi_i =1\}$. If $\kappa_0 =-2$ then we set $\eta = \eta_0$.  Otherwise, we let $\gamma_0$ be such that $a_{\kappa_0} \meet a_{\gamma_0}.$ We then define $\eta_1> \eta_0$ as
%%
% \begin{align} \eta_1 = \begin{cases}  {\gamma_0}, & X_{\gamma_0} =  -1 \\ {\gamma_0} + (\gamma_0 -\kappa_0) -1, &  X_{\gamma_0} = 0\end{cases}. \label{eq:eta1}
%\end{align}
% Let $\xi^{(1)} = \xi( \X[0,\eta_1]) = (\xi^{(1)}_0, \hdots \xi^{(1)}_{\eta_1})$ and $\kappa_1 = \min \{ i \colon \xi_1^{(1)} = 1\}$. If $\kappa_1 = -2$ then set $\eta = \eta_1$. Otherwise, let $\gamma_1$ be such that $a_{\kappa_1} \meet a_{\gamma_1}$. We then define 
% \begin{align*} \eta_2 = \begin{cases}  {\gamma_1}, & X_{\gamma_1} =  -1 \\ {\gamma_1} + (\gamma_1- \kappa_1) -1, &  X_{\gamma_1} = 0\end{cases}.
%\end{align*}
%We then perform the analogous checks to the previous case. We let $\xi^{(2)} = \xi( \X[0,\eta_2])$ and $\kappa_2 = \min \{ i \colon \xi^{(2)}_i = 1\}$. If there are no surviving $1$-particles then set $\eta = \eta_2$, otherwise we continue in this manner. If this iterative process never terminates ($\eta = \infty$) then there is always a surviving $1$-particle to the right of the seed. This implies that the seed is never annihilated. If the process terminates and $\eta = \eta_j$ for some $j < \infty$, then the leftmost surviving 1-particle from $\BA(\X[0, \eta_{j-1}]$ is annihilated by a $-1$-particle. By \thref{lem:active} the speeds $(X_i)_{i > \eta}$ are independent. 
% \end{proof}

Define the random variable $Z = |\{ i \colon \xi_i(\X[0,\eta]) = 0\}|$ to be the number of surviving inert particles. Because the process renews after $\eta$ we can link the expected number of surviving inert particles in $\BA(\X[0,\eta])$ to $\psi(p)$ via a Galton-Watson process.

\begin{prop} \thlabel{prop:iff}
 Let $\eta$ be as in \thref{prop:eta} and $Z \geq 0 $ be the number of surviving inert particles from $\BA(\X[0,\eta])$. It holds that $$\psi(p) >0 \iff \E Z >1.$$
\end{prop}

\begin{proof}
 The random variable $Z$ can be used as the offspring distribution for a Galton-Watson process that counts surviving inert particles. Starting with the seed, \thref{prop:eta} ensures that we have $Z$ inert particles in $[0,\eta]$ that survive from $\BA(\X[0,\eta])$. Moreover, the speeds $(X_i)_{i > \eta}$ are i.i.d.\ $\nu$-distributed. 
 
 We claim that each of these $Z$ particles will (eventually) serve as a seed that spawns $Z$-distributed more surviving inert particles. We do this in a ``depth first" manner. Consider the rightmost surviving inert particle in $[0,\eta]$. Say it is at $i$. By construction there are no surviving active particles in $[i,\eta]$. If $i= \eta$, then this is exactly the same initial configuration as with the seed. Even if $i<\eta$, we still obtain a $Z$-distributed number of surviving inert particles before the process renews. This is because a $-1$-particle must destroy the inert particle at $i$. This happens irregardless of whether or not there is a gap $(i < \eta)$ or not ($i =\eta)$. We continue generating $Z$ offspring at each surviving inert particle and stop if the process goes extinct. 

The previous discussion implies that if the $Z$-distributed Galton-Watson process survives then the seed is never annihilated. Conversely, if $\psi(p) >0$, then 
the seed survives with positive probability. This can only happen if the Galton-Watson process started from the seed does not going extinct. The result follows from the elementary fact that non-extinction of a Galton-Watson process with positive probability is equivalent to $\E Z >1$.
\end{proof}

\section{Approximating $\E Z$}
The complicated definition of $\eta$ from \thref{prop:eta} suggests it would be difficult to calculate $\E Z$ explicitly. Even calculating the distribution of $\eta$ seems beyond our reach. However, there are certain times before $\eta$ at which we can obtain lower bounds on $\E Z$.  It helps to explain our approach in stages. First, we consider the effect of $a_1$.

\subsection{The effect of $\BA(\X[0,1])$}  The simplest lower bound on $Z$ is to look at what happens in $\BA(\X[0,1])$. If $X_1 = -1$, then $Z= 0$, and if $X_1 = 0$, then $Z=2$. This gives 
\begin{align}
Z \succeq \ind{X_1 = 0}2.	\label{eq:Z00}
\end{align}
Thus, if $p > 1/2$, we have $\E Z >1$. Equivalently, $p_c' \leq 1/2$. This is a start, but not so interesting. The same statement could be proven by accounting for the number of inert particles versus active particles with a $p$-biased random walk. Survival of the seed is equivalent to the walk never returning to 0, which happens with positive probability when $p >1/2$. 

Let us go one step further by considering the case $\eta>1$. If $X_1 = 1$, then the seed will not be annihilated by any particles in $[0,\eta]$ by construction of $\eta$. Thus, $Z \succeq 1$ on this event.   This gives the more meaningful bound
\begin{align}Z \succeq  \ind{X_1 = 0}2 + \ind{X_1 = 1}.\label{eq:Z0}\end{align}
Taking expectation in \eqref{eq:Z0} yields $\E Z \geq 2p + \f{1-p}2.$
When $p > 1/3$, this is larger than $1$. Thus, we arrive easily at the bound $p_c' \leq 1/3$, which is proven in \cite{bullets} and proven for $p_c$ in \cite{arrows}.

\subsection{The effect of $a_{\gamma_1}$} \label{sec:step2}

We can do better by extracting some benefit from the $1$-particle at $1$. Let $\gamma_1$ be the index of the particle that destroys $a_1$ as in the proof of \thref{prop:eta}. Let $Z_1$ be the number of surviving inert particles in the window of dependence induced by the event $\{a_1 \meet a_{\gamma_1}\}$. Depending on whether $X_{\gamma_1}= -1$ or $0$, this window is either $[0,\gamma_1]$ or $[0,2\gamma_1 -1]$. By construction, these inert particles will not be destroyed in $\BA(\X[0,\eta])$,  and thus $\ind{X_1=1} Z_1 \preceq \ind{X_1 = 1}Z$. This lets us improve \eqref{eq:Z0} to the following dominance relation
\begin{align}
Z \succeq \ind{X_1 = 0} 2 + \ind{ X_1 = 1} Z_1.	\label{eq:better}
\end{align}
 Figure \ref{fig:Z1} depicts a realization in which $Z_1 = 2$. 
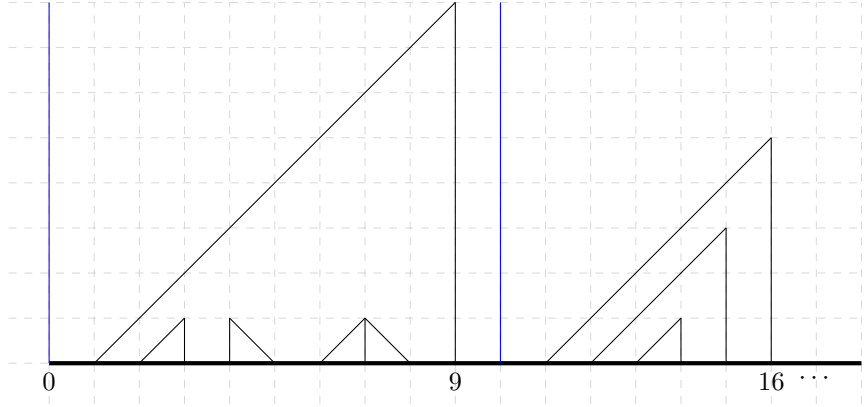
\begin{figure}
\begin{center}

\begin{tikzpicture}[scale = .6]

	\draw[help lines, color=gray!30, dashed] (-0.9,-.9) grid (18,8);
	\draw[ultra thick] (0,0)--(18,0);
	\draw[blue] (0,0) -- (0,8);
	\node[below] at (0,0) {$0$};
	\draw (1,0) -- (9,8);
%	\node[below] at (1,0) {$1$};
	\draw (2,0) -- (3,1);
%	\node[below] at (2,0) {$2$};
	\draw (3,0) -- (3,1);
%	\node[below] at (3,0) {$3$};
	\draw (4,0) -- (4,1);
	\draw (5,0) -- (4,1);
	\draw (6,0) -- (7,1);
	\draw (7,0) -- (7,1);
	\draw (8,0) -- (7,1);
	\draw (9,0) -- (9,8);
	\draw[blue] (10,0) -- (10,8);
	\draw (11,0) -- (16,5);
	\draw (12,0) -- (15,3);
	\draw (13,0) -- (14,1);
	\draw (14,0) -- (14,1);
	\draw (15,0) -- (15,3);
	\draw (16,0) -- (16,5);
%	\node[below] at (4,0) {$4$};
%	\node[below] at (6.5 ,0) {$\cdots$};
%	\draw (9,0) -- (6.5,2.5);
	\node[below] at (9,0) {$9$};
	\node[below] at (16,0) {$16$};
	\node[below] at (17 ,0) {$\cdots$}; 
%	\node[above = .25 cm] at (17,0) {i.i.d.};

\end{tikzpicture} 
\caption{In this example, $\gamma_1 =9$. We are left with $Z_1 = 2$ surviving inert particles in $\BA(\X[0,2\gamma_1 - 1])$.} \label{fig:Z1}
\end{center}
\end{figure}
To better understand $Z_1$, we decompose it relative to the behavior of $\gamma_1$:
\begin{align}Z_1 &= \sum_{n=2}^\infty  \ind{\gamma_1 = n} Z_1 \nonumber\\
&= \sum_{n=2}^\infty  \ind{X_{\gamma_1} = -1 , \gamma_1 = n} + \ind{ X_{\gamma_1} = 0,\gamma_1 = n} Z_1 \nonumber \\
&= 1 + \sum_{n=2}^\infty \ind{X_{\gamma_1} = 0, \gamma_1 = n} (Z_1-1).
\label{eq:EZ1}
	\end{align}
	Note that $\gamma_1 \geq 2$ since we only introduce $\gamma_1$ when $X_1 = 1$, and so the particle that destroys $a_1$ must start at $2$ or beyond. All of the cases where $X_{\gamma_1} = -1$ are included in the leading $1$ summand, so we devote our attention to when $X_{\gamma_1} = 0$.
%	Also, observe that $Z_1 \geq 1$ since the seed always survives in $\BA([0,\eta])$ when $X_1 =1$. Thus, the term $Z_1 -1$ is nonnegative.
%	\HOX{Need to say what $A_n$ is a subset of. -MJ}
	Let 
		\begin{align}
	A_n = \{X_{\gamma_1} = 0,\gamma_1 = n\} \subseteq \{0\} \times \{1\} \times \{-1,0,1\}^{2 n - 2}\label{eq:An}
		\end{align}
be the set of sub-configurations for which $\gamma_1 =n$ and $X_{\gamma_1} = 0$. Notice we only need to know the entries up to $2 n -1$ to determine if $a_{1} \mapsto a_{n}$.

In light of \eqref{eq:EZ1}, we would like to understand
	$\E[\ind{A_n}(Z_1-1)]$. An important observation is that $Z_1 -1 \geq 0$ since the seed always survives in $\BA([0, 2 \gamma_1 - 1])$. This means we can lower bound the expectation by computing its value for finitely many $n$. We do so by counting surviving inert particles in the final state of $\BA(\X[0,2 \gamma_1 -1])$ in each of the $ 3^{2n-2}$ possible realizations of particle speeds from $\{0\} \times \{1\}\times  \{ -1, 0 ,1\}^{2n-2}$. 
%Let $G_n = \{\vec x \in \Omega_{2n-2} \colon \gamma_0 = n\} $ be the set of all configurations that result in $\{\gamma_0 = n\}$. 

Given $\vec x \in A_n$, let $I(\vec x)$ be the number of $0$-entries in $\vec x$. Simply by the definition of $\nu$ at \eqref{eq:nu}, the probability $\vec x$ occurs is
$$q_n(\vec x) := p^{I(\vec x)} \left(\f{1-p}{2} \right)^{2n-2 - I(\vec x)}.$$
Letting $Z_1(\vec x)$ be the number of surviving inert particles in $\BA(\vec x)$, we then have
\begin{align}\E[ \ind{A_n} (Z_1 -1)]= \sum_{\vec x \in A_n} q_n(\vec x) [ Z_1(\vec x) -1] .\label{eq:Pg}	
\end{align}

This is possible for a computer to calculate for small values of $n$. For instance, if we compute for $n\leq 18$ (approximately 400 million cases) and take expectation in \eqref{eq:EZ1}, we obtain the bound
%\HOX{Actually find $m$ here. -MJ}
\begin{align}
\E Z_1 \geq 1 + \sum_{n=2}^{18} \sum_{x \in A_n} q_n(\vec x) [ Z_1(\vec x) -1] = m(p),\label{eq:m(p)}
\end{align}
and thus by \eqref{eq:better} we have
\begin{align}\E Z \geq 2p + \f{1-p}2 m(p).\label{eq:EZ1'}
\end{align}
By numerically checking the boundary values of $p$, we find that $\E Z>1$ when $p > 0.2914$. Note that $m(.2914) \approx 1.1178$. So the ``gain" we had from the previous calculation is $\approx .1178$ more expected surviving inert particles.

\subsection{Using a surviving $1$-particles from $\BA(\X[0,2 \gamma_1 - 1])$}
A simple way to optimize further is by re-using the previous calculation for configurations from $A_n$ in which there is a single surviving  $1$-particle at $2n-1$, and otherwise only inert particles. Formally, let 
$$A_n' = \{ \vec x \in A_n \colon \xi_{2n-1}(\vec x) = 1, \xi_i(\vec x) \in \{0,2\} \text{ for }0 \leq i \leq 2n-2\}$$
 be the set of all such configurations. We claim that each $\vec x' \in A_n'$ provides an independent $Z_1-1$ distributed number of inert particles. This is because the particle speeds to the right of $a_{2\gamma_1-1}$ are independent. Thus, we can couple the number of surviving inert particles induced by $a_{2 \gamma_1 -1}$ to $Z_1-1$. We subtract 1 so we do not double count the seed. Let $b(p) = \sum_{n=1}^{18}\sum_{\vec x ' \in A'_n} q_n(\vec x ')$ be the probability of a configuration from $A_n'$. Each time this occurs we obtain an expected $m(p) -1$ more inert particles with $m(p)$ from \eqref{eq:m(p)}. This will happen geometric($b(p)$)-distributed many times, which has expectation $b(p) / (1- b(p))$. It follows that we can improve our bound from \eqref{eq:EZ1'} to the following  
 $$\E Z \geq 2p + \f{1-p}{2}\left( m(p) + \f{b(p)}{1- b(p)}(m(p)-1)\right).$$
It takes a computer about three hours to obtain an algebraic expression for $b(p)$. After doing so, we find that $\E Z >1$ when $p > .2870.$ In this case we have $m(p) \approx 1.1713$ and $b(p) \approx .1226.$

\subsection{Further benefit from surviving active particles} The bound we obtain on $p_c'$ is as far as seemed reasonable to push our technique. More complicated calculations can be done where one considers the impact of other configurations of surviving $1$-particles to the right of $\gamma_1$. For example, one might consider the case that there is a single surviving $1$-particle at $2 \gamma_1 -k$ with $k \geq 1$ (we only consider $k=1$). However, we found that the improvements to our bound were very small (about a $.004$ to our bound on $p_c$) did not justify the added complexity to the argument. 
%As discussed in the introduction, since it is conjectured that the expected lifetime of a $1$-particle diverges as $p \to p_c$, i
It is our belief that the main benefit to estimating $\E Z$ would come out of extending our approach and looking out to distances $n >18$. When $p=.2870$ we exactly compute $P(\gamma_1 \leq 18) = .9018$. So, we are missing about $1/10$ of the right tail, which ought to contain a significant number of surviving inert particles. However, without a clever insight, extending  much further appears computationally intractable.

\section*{Acknowledgements}

We thank Rick Durrett and Laurent Tournier for invaluable comments and discussion.

\bibliographystyle{amsalpha}
\bibliography{ballistic}

\end{document}